\documentclass{amsart}
\usepackage[initials]{amsrefs}
\usepackage{graphicx,amscd,color,amsmath,amsfonts,amssymb,geometry}
\newtheorem{theorem}{Theorem}[section]
\newtheorem{claim}[theorem]{Claim}

\numberwithin{equation}{section}

\begin{document}
%%%%%%%
\title[On the constants in the Bohnenblust--Hille inequality]{The asymptotic growth of the constants in the Bohnenblust-Hille inequality is optimal}
%%%%%%%
\author[D. Diniz \and G. A. Mu\~{n}oz-Fern\'{a}ndez \and D. Pellegrino \and J. B. Seoane-Sep\'{u}lveda]{D. Diniz, G. A. Mu\~{n}oz-Fern\'{a}ndez
\textsuperscript{*} \and D. Pellegrino\textsuperscript{**} \and J. B. Seoane-Sep\'{u}lveda\textsuperscript{*}}
%%%%%%%
\address{Unidade Academica de Matem\'{a}tica e Estat\'{i}stica,\newline\indent Universidade Federal de Campina Grande, \newline\indent Caixa Postal 10044, \newline\indent Campina Grande, 58429-970, Brazil.}
\email{diogo@dme.ufcg.edu.br}
%%%%%%%
\address{Departamento de An\'{a}lisis Matem\'{a}tico,\newline\indent Facultad de Ciencias Matem\'{a}ticas, \newline\indent Plaza de Ciencias 3, \newline\indent Universidad Complutense de Madrid,\newline\indent Madrid, 28040, Spain.}
\email{gustavo$\_$fernandez@mat.ucm.es}
%%%%%%%
\address{Departamento de Matem\'{a}tica, \newline\indent Universidade Federal da Para\'{\i}ba, \newline\indent 58.051-900 - Jo\~{a}o Pessoa, Brazil.} \email{pellegrino@pq.cnpq.br}
%%%%%%%
\address{Departamento de An\'{a}lisis Matem\'{a}tico,\newline\indent Facultad de Ciencias Matem\'{a}ticas, \newline\indent Plaza de Ciencias 3, \newline\indent Universidad Complutense de Madrid,\newline\indent Madrid, 28040, Spain.}
\email{jseoane@mat.ucm.es}
%%%%%%%
%%%%%%%
\thanks{\textsuperscript{*}Supported by the Spanish Ministry of Science and Innovation, grant MTM2009-07848.}
%%%%%%%
\thanks{\textsuperscript{**}Supported by CNPq and PROCAD Novas Fronteiras CAPES}
%%%%%%%
%%%%%%%
\begin{abstract}
We provide (for both the real and complex settings) a family of constants, $%
(C_{m})_{m\in \mathbb{N}}$, enjoying the Bohnenblust--Hille inequality and
such that $\displaystyle\lim_{m\rightarrow \infty }\frac{C_{m}}{C_{m-1}}=1$,
i.e., their asymptotic growth is the best possible. As a consequence, we
also show that the optimal constants, $(K_{m})_{m\in \mathbb{N}}$, in the
Bohnenblust--Hille inequality have the best possible asymptotic behavior.
Besides its intrinsic mathematical interest and potential applications to
different areas, the importance of this result
also lies in the fact that all previous estimates and related results for
the last 80 years (such as, for instance, the multilinear version of the
famous Grothendieck Theorem for absolutely summing operators) always present
constants $C_{m}$'s growing at an exponential rate of certain power of $m$.
\end{abstract}
\maketitle

\section{Preliminaries. The History of the Problem}

The Bohnenblust--Hille inequality (1931, \cite{bh}) asserts that for every
positive integer $m\geq2$ there exists a sequence of non decreasing positive
scalars $\left( C_{m}\right) _{m=1}^{\infty }$ such that
\begin{equation}
\left( \sum\limits_{i_{1},\ldots ,i_{m}=1}^{N}\left\vert
U(e_{i_{^{1}}},\ldots ,e_{i_{m}})\right\vert ^{\frac{2m}{m+1}}\right) ^{%
\frac{m+1}{2m}}\leq C_{m}\sup_{z_{1},...,z_{m}\in \mathbb{D}^{N}}\left\vert
U(z_{1},...,z_{m})\right\vert  \label{hypp}
\end{equation}%
for all $m$-linear mapping $U:\mathbb{C}^{N}\times \cdots \times \mathbb{C}%
^{N}\rightarrow \mathbb{C}$ and every positive integer $N$, where $\left(
e_{i}\right) _{i=1}^{N}$ denotes the canonical basis of $\mathbb{C}^{N}$ and
$\mathbb{D}^{N}$ represents the open unit polydisk in $\mathbb{C}^{N}$. The
original constants obtained by Bohnenblust and Hille are
\begin{equation*}
C_{m}=m^{\frac{m+1}{2m}}2^{\frac{m-1}{2}}.
\end{equation*}%
Making $m=2$ in (\ref{hypp}) we recover Littlewood's $4/3$ inequality (1930,
\cite{Litt}). In the last 80 years, very few improvements for the constants $%
C_{m}$ have been achieved:

\begin{itemize}
\item $C_{m}=2^{\frac{m-1}{2}}$ (Kaijser \cite{Ka} and Davie \cite{d}),

\item $C_{m}=\left( \frac{2}{\sqrt{\pi }}\right) ^{m-1}$ (Queff\'{e}lec \cite%
{Q} and Defant \& Sevilla-Peris \cite{d2}).
\end{itemize}

Besides their intrinsic mathematical interest, the Bohnenblust--Hille
inequality and Littlewood's $4/3$ theorem are extremely useful tools in many
areas of Mathematics, just to cite some: Operator theory in Banach spaces,
Fourier and harmonic analysis, analytic number theory, etc. (we refer the
interested reader to the monographs \cite{Bl, Ko, Pi}).

More recently a new proof of the Bohnenblust--Hille inequality was presented
in \cite{defant} and this approach, although very abstract, allowed the
calculation of new and sharper constants (see \cite{ps}). However all the
possible constants derived from this approach were calculated by means of
recursive formulae, and the expression of these constants as closed formulae
seems to be, in most cases, an impossible task. The polynomial version of
Bohnenblust--Hille inequality, due to its different nature, presents worse
constants and only in 2011 the long standing problem on the
hypercontractivity of the constants for the polynomial Bohnenblust--Hille
inequality was settled in \cite{ann}, i.e., the authors proved that (for the
polynomial case) there exists $C>1$ so that
\begin{equation*}
\frac{C_{m}}{C_{m-1}}=C
\end{equation*}%
for all $m\geq 1$.

Notwithstanding its eight decades of existence of the Bohnenblust--Hille
inequality, the optimal asymptotic behavior of the constants involved is still unknown. From the previous estimates
we have

\begin{itemize}
\item $C_{m}/C_{m-1}=\sqrt{2}\approx1.4142$ (Kaisjer \cite{Ka} and Davie
\cite{d}),

\item $C_{m}/C_{m-1}=\frac{2}{\sqrt{\pi}}\approx1.1284$ (Queff\'{e}lec \cite%
{Q} and Defant \& Sevilla-Peris \cite{d2}).
\end{itemize}

A very recent numerical study (\cite{mps}), presented the possibility of
improving the above results (now including the case of real scalars as well)
to
\begin{equation*}
\displaystyle\lim_{m\rightarrow \infty }\frac{C_{m}}{C_{m-1}}=2^{1/8}\approx
1.1090,
\end{equation*}%
but this estimate could not be formally proved.

In this article we prove an optimal result for the asymptotic behavior of
the Bohnen\-blust--Hille constants: we provide a family of constants, $%
(C_{m})_{m\in \mathbb{N}}$, satisfying the Bohnenblust--Hille inequality
with the best possible asymptotic growth, i.e.,
\begin{equation*}
\displaystyle\lim_{m\rightarrow \infty }\frac{C_{m}}{C_{m-1}}=1.
\end{equation*}

As a consequence we conclude that the optimal constants in the
Bohnenblust--Hille inequality have an optimal behavior, asymptotically
speaking.

As we mentioned before, for real scalars the Bohnenblust--Hille inequality
is also true; for a long time the best known constants in this setting
seemed to be $C_{m}=2^{\frac{m-1}{2}}.$ For $m=2$ we have $C_{2}=\sqrt{2}$,
which is known to be optimal. Thus, it is possible that this absence of
improvements for the case of real scalars is motivated by a feeling that the
family $C_{m}=2^{\frac{m-1}{2}}$ was a good candidate for being optimal.
However, as we will see, these constants are quite far from the optimality.

From now on, the letters $X_{1},\ldots ,X_{m}$ shall stand for Banach spaces
and $B_{X^{\ast }}$ represents the closed unit ball of the topological dual
of $X$. By $\mathcal{L}(X_{1},\ldots ,X_{m};\mathbb{K})$ we denote the
Banach space of all continuous $m$-linear mappings from $X_{1}\times \cdots
\times X_{m}$ to $\mathbb{K}$ endowed with the sup norm.

If $1\leq p<\infty $, a continuous $m$-linear mapping $U\in \mathcal{L}%
(X_{1},\ldots ,X_{m};\mathbb{K})$ is called multiple $(p;1)$-summing
(denoted by $U\in \Pi _{(p;1)}(X_{1},\ldots ,X_{m};\mathbb{K})$) if there
exists a constant $K_{m}\geq 0$ such that
\begin{equation}
\left( \sum_{j_{1},\ldots ,j_{m}=1}^{N}\left\vert U(x_{j_{1}}^{(1)},\ldots
,x_{j_{m}}^{(m)})\right\vert ^{p}\right) ^{\frac{1}{p}}\leq
K_{m}\prod_{k=1}^{m}\sup_{\varphi _{k}\in B_{X_{k}^{\ast }}}{%
\sum\limits_{j=1}^{N}}\left\vert \varphi _{k}(x_{j}^{(k)})\right\vert
\label{lhs}
\end{equation}%
for every $N\in \mathbb{N}$ and any $x_{j_{k}}^{(k)}\in X_{k}$, $%
j_{k}=1,\ldots ,N$, $k=1,\ldots ,m$. The infimum of the constants satisfying
(\ref{lhs}) is denoted by $\left\Vert U\right\Vert _{\pi (p;1)}$.

An important simple reformulation of the Bohnenblust--Hille inequality shows
that it is equivalent to the assertion that every continuous $m$-linear form
$T:X_{1}\times\cdots\times X_{m}\rightarrow\mathbb{K}$ is multiple $(\frac {%
2m}{m+1};1)$-summing, and the constants involved are the same as those from
the Bohnenblust--Hille inequality. Thus, in this article we shall be dealing
with the Bohnenblust--Hille inequality in the framework of multiple summing
operators.

From now on (and throughout this paper), we let
\begin{equation*}
A_{p}:=\sqrt{2}\left( \frac{\Gamma((p+1)/2)}{\sqrt{\pi}}\right) ^{1/p},
\end{equation*}
where $\Gamma$ denotes the classical Gamma Function. These constants will
appear soon in Theorems 1.1 and 1.2; they are related to the Khinchine
inequality and appear in \cite{haag}, where the optimal constants of the
Khinchine inequality are obtained. More details on the nature of $A_{p}$ will be given in the next section.

The following results (inspired in \cite{defant} and in estimates from \cite%
{haag}) were recently proved in \cite{ps} by the third and fourth authors:

\begin{theorem}
\label{p1}For every positive integer $m$ and real Banach spaces $%
X_{1},\ldots,X_{m},$
\begin{equation*}
\Pi_{(\frac{2m}{m+1};1)}(X_{1},\ldots,X_{m};\mathbb{R})=\mathcal{L}%
(X_{1},...,X_{m};\mathbb{R})\text{ and }\left\Vert .\right\Vert _{\pi (\frac{%
2m}{m+1};1)}\leq C_{\mathbb{R},m}\left\Vert .\right\Vert
\end{equation*}
with%
\begin{align*}
C_{\mathbb{R},2} & =\sqrt{2}, \\
C_{\mathbb{R},3} & =2^{\frac{5}{6}},
\end{align*}%
\begin{equation*}
C_{\mathbb{R},m}=\frac{C_{\mathbb{R},m/2}}{A_{\frac{2m}{m+2}}^{m/2}}
\end{equation*}
for $m$ even and%
\begin{equation*}
C_{\mathbb{R},m}=\left( \frac{C_{\mathbb{R},\frac{m-1}{2}}}{A_{\frac {2m-2}{%
m+1}}^{\frac{m+1}{2}}}\right) ^{\frac{m-1}{2m}}\cdot\left( \frac{C_{\mathbb{R%
},\frac{m+1}{2}}}{A_{\frac{2m+2}{m+3}}^{\frac{m-1}{2}}}\right) ^{\frac{m+1}{%
2m}}
\end{equation*}
for $m$ odd.
\end{theorem}

\begin{theorem}
\label{p2}For every positive integer $m$ and every complex Banach spaces $%
X_{1},\ldots,X_{m},$
\begin{equation*}
\Pi_{(\frac{2m}{m+1};1)}(X_{1},\ldots,X_{m};\mathbb{C})=\mathcal{L}%
(X_{1},\ldots,X_{m};\mathbb{C})\text{ and }\left\Vert .\right\Vert _{\pi (%
\frac{2m}{m+1};1)}\leq C_{\mathbb{C},m}\left\Vert .\right\Vert
\end{equation*}
with%
\begin{equation*}
C_{\mathbb{C},m}=\left( \frac{2}{\sqrt{\pi}}\right) ^{m-1}
\end{equation*}
$\text{for }m\in\{2,3,4,5,6\},$%
\begin{equation*}
C_{\mathbb{C},m}=\frac{C_{\mathbb{C},m/2}}{A_{\frac{2m}{m+2}}^{m/2}}
\end{equation*}
for $m>6$ even and%
\begin{equation*}
C_{\mathbb{C},m}=\left( \frac{C_{\mathbb{C},\frac{m-1}{2}}}{A_{\frac {2m-2}{%
m+1}}^{\frac{m+1}{2}}}\right) ^{\frac{m-1}{2m}}\cdot\left( \frac{C_{\mathbb{C%
},\frac{m+1}{2}}}{A_{\frac{2m+2}{m+3}}^{\frac{m-1}{2}}}\right) ^{\frac{m+1}{%
2m}}
\end{equation*}
for $m>5$ odd.
\end{theorem}

In this article we shall prove that the asymptotic behavior of the above
constants is the best possible. This result (the optimality of the
asymptotic growth of the constants of Bohnenblust--Hille inequality) seems
quite surprising since all the constants involved in all similar results in
the theory of multiple summing operators grow at an exponential rate. In
fact, this is the case of the previous estimates of the constants of
Bohnenblust--Hille inequality from \cite{d, d2, Ka, Q} and also the case of
the multilinear generalization (for multiple summing operators) of
Grothendieck's theorem for absolutely summing operators. More precisely, the
multilinear Grothendieck Theorem \cite{Bo} states that for every positive
integer $m$, every continuous $m$-linear operator $U:\ell _{1}\times \cdots
\times \ell _{1}\rightarrow \mathbb{C}$ is multiple $(1;1)$-summing and
\begin{equation*}
\left\Vert U\right\Vert _{\pi (1;1)}\leq \left( \frac{2}{\sqrt{\pi }}\right)
^{m}\left\Vert U\right\Vert .
\end{equation*}%
Similar results, which can also be found in \cite{Bo} (called \emph{%
coincidence} results), always present constants that, as we mentioned
earlier, grow exponentially following the pattern of certain power of $m;$
this is one of the reasons why we consider the optimality of the asymptotic growth
of the constants in the Bohnenblust--Hille inequality quite a surprising
result.

%%%%%%%
%%%%%%%
%%%%%%%
%%%%%%%

\section{Bohnenblust and Hille meet Euler}

In this section we shall see that the limits of the terms that appear in the
previous sections are related to Euler's famous constant $\gamma$. Let us
recall that $\gamma$ is classically defined as
\begin{equation*}
\gamma= \lim_{m \rightarrow\infty} \left( \sum_{k=1}^{m} \frac{1}{k} - \log
m\right) \approx 0.5772,
\end{equation*}
and it is still not known (since its first appearance in 1734 in Euler's
\emph{De Progressionibus Harmonicis Observationes}) whether this value is
algebraic or transcendental.

The results and calculations given in this section are part of the essential
tools that shall be used to prove the main result of this article. First of
all, and by using the classical known properties of the Gamma Function, we
obtain
\begin{equation*}
\lim_{x\rightarrow 0}\left( \frac{\Gamma \left( \frac{1}{2}-x\right) }{%
\Gamma \left( \frac{1}{2}\right) }\right) ^{1/x}=4e^{\gamma },
\end{equation*}%
and, thus,
\begin{equation*}
\lim_{x\rightarrow 0}\left( \frac{\Gamma \left( \frac{3}{2}-x\right) }{%
\Gamma \left( \frac{3}{2}\right) }\right) ^{1/x}=4e^{\gamma -2}
\end{equation*}%
which gives
\begin{equation*}
\lim_{m\rightarrow \infty }\left( \frac{\Gamma \left( \frac{3m+2}{2m+4}%
\right) }{\Gamma \left( \frac{3}{2}\right) }\right) ^{m}=16e^{2\gamma -4},
\end{equation*}%
obtaining
\begin{equation*}
\lim_{m\rightarrow \infty }\left[ 2^{m/4}\cdot \left( \frac{\Gamma \left(
\frac{\frac{2m}{m+2}+1}{2}\right) }{\sqrt{\pi }}\right) ^{(m+2)/4}\right] =%
\frac{\sqrt{2}}{e^{1-\frac{1}{2}\gamma }}.
\end{equation*}%
From \cite{haag} it is known that for $1<p_{0}<2$ so that
\begin{equation*}
\Gamma \left( \frac{p_{0}+1}{2}\right) =\frac{\sqrt{\pi }}{2}
\end{equation*}%
we have
\begin{equation*}
A_{p}=\sqrt{2}\left( \frac{\Gamma ((p+1)/2)}{\sqrt{\pi }}\right) ^{1/p},
\end{equation*}%
whenever $p_{0}\leq p<2$ (numerical calculations estimate $p_{0}\approx
1.8474$). So, since
\begin{equation*}
\lim_{m\rightarrow \infty }\frac{2m}{m+2}=2>p_{0},
\end{equation*}%
we finally get that (for both, the real and complex settings)
\begin{equation}
\lim_{m\rightarrow \infty }\left( A_{\frac{2m}{m+2}}^{m/2}\right) ^{-1}=%
\frac{e^{1-\frac{1}{2}\gamma }}{\sqrt{2}}\approx 1.4402.  \label{jt}
\end{equation}%
Analogously, it can be checked that (for $m$ odd)
\begin{equation}
\lim_{m\rightarrow \infty }\left( A_{\frac{2m-2}{m+1}}^{(m+1)/2}\right)
^{(-1)\cdot \frac{m-1}{2m}}=\lim_{m\rightarrow \infty }\left( A_{\frac{2m+2}{%
m+3}}^{(m-1)/2}\right) ^{(-1)\cdot \frac{m+1}{2m}}=\frac{e^{\frac{1}{2}-%
\frac{1}{4}\gamma }}{2^{1/4}}\approx 1.2001.  \label{gt}
\end{equation}%
%
%%%%%%%
%%%%%%%
%%%%%%%
%%%%%%%
%%%%%%%

%%%%%%%
%%%%%%%

\section{The Proof}

\noindent From (\ref{jt}) and (\ref{gt}) we have (for both, the real and
complex settings) that
\begin{equation*}
\displaystyle\lim_{n\rightarrow \infty }\frac{C_{2n}}{C_{n}}=\frac{e^{1-%
\frac{1}{2}\gamma }}{\sqrt{2}}
\end{equation*}%
and
\begin{equation*}
\displaystyle\lim_{n\rightarrow \infty }\frac{C_{2n+1}}{\left( C_{n}\right)
^{\frac{2n}{2\left( 2n+1\right) }}.\left( C_{n+1}\right) ^{\frac{2n+2}{%
2\left( 2n+1\right) }}}=\left( \frac{e^{\frac{1}{2}-\frac{1}{4}\gamma }}{%
2^{1/4}}\right) .\left( \frac{e^{\frac{1}{2}-\frac{1}{4}\gamma }}{2^{1/4}}%
\right) =\frac{e^{1-\frac{1}{2}\gamma }}{\sqrt{2}}.
\end{equation*}%
In other words, for \emph{large} values of $n$ we have the following
equivalences
\begin{align*}
C_{2n}& \sim \left( \frac{e^{1-\frac{1}{2}\gamma }}{\sqrt{2}}\right) C_{n},
\\
C_{2n+1}& \sim \left( \frac{e^{1-\frac{1}{2}\gamma }}{\sqrt{2}}\right)
\left( C_{n}\right) ^{\frac{n}{2n+1}}\left( C_{n+1}\right) ^{\frac{n+1}{2n+1}%
},
\end{align*}%
as well from Theorem 1.1, Theorem 1.2, (\ref{jt}) and (\ref{gt}) we have the
next asymptotical identities:
\begin{equation}
\frac{C_{2n}}{C_{2n-1}}\sim \frac{C_{n}}{\left( C_{n-1}\right) ^{\frac{n-1}{%
2n-1}}\left( C_{n}\right) ^{\frac{n}{2n-1}}}\sim \left( \frac{C_{n}}{C_{n-1}}%
\right) ^{\frac{n-1}{2n-1}}  \label{add44}
\end{equation}%
and
\begin{equation}
\frac{C_{2n+1}}{C_{2n}}\sim \frac{\left( C_{n}\right) ^{\frac{n}{2n+1}%
}\left( C_{n+1}\right) ^{\frac{n+1}{2n+1}}}{C_{n}}\sim \left( \frac{C_{n+1}}{%
C_{n}}\right) ^{\frac{n+1}{2n+1}}.  \label{add55}
\end{equation}

\noindent Since $\displaystyle \left( \frac{C_{2n}}{C_{n}}\right)
_{n=1}^{\infty}$ converges, it is bounded by some constant $C>1$. Let, for $%
n \in \mathbb{N}$,
\begin{equation*}
D_{n}:=\frac{C_{n+1}}{C_{n}}.
\end{equation*}
Notice also that, since $\left(C_{n}\right)_{n=1}^{\infty}$ is
non-decreasing, $D_{n}\geq 1$. Thus
\begin{equation}
C >\frac{C_{2n}}{C_{n}}={\textstyle\prod\limits_{j=n}^{2n-1}}D_{j}\geq
D_{n}\geq 1  \label{di1}
\end{equation}
for every $n$.

Now we show that $\left( D_{n}\right) _{n=1}^{\infty}$ is convergent and
that its limit must be $1$. We split its proof into four \emph{claims}. The
first one of these claims is an almost trivial statement, and we spare the
details of its proof to the interested reader:

\begin{claim}
\label{claim1} The sequence $\left( D_{n}\right) _{n=1}^{\infty}$ satisfies
the following asymptotical equalities:

\begin{enumerate}
\item[(1)] $D_{2n-1}\sim\sqrt{D_{n-1}}$,

\item[(2)] $D_{2n}\sim\sqrt{D_{n}}$.
\end{enumerate}
\end{claim}

The second of the claims is a crucial tool in order to prove that $%
\left(D_{n}\right)_{n=1}^{\infty}$ converges to $1$:

\begin{claim}
\label{claim2} Let $K>1$. Suppose that there exists $m_{0}$ such that $%
n>m_{0}$ implies $1\leq D_{n}<K$. Then there exists $m_{1} \geq m_{0}$ such
that $n > m_{1}$ implies $1\leq D_{n}<K^{5/8}$.
\end{claim}

\begin{proof}
From (1) in Claim \ref{claim1} we know that $\displaystyle %
\lim_{n\rightarrow\infty}\frac{D_{2n-1}}{\sqrt{D_{n-1}}}=1.$ Given $%
\varepsilon>0$ satisfying $\varepsilon<K^{\frac{1}{8}}-1$ there exists $%
n_{0}>m_{0}+1$ such that $n>n_{0}$ implies
\begin{equation*}
\frac{D_{2n-1}}{\sqrt{D_{n-1}}}<(1+\varepsilon)<K^{1/8}.
\end{equation*}
Therefore
\begin{equation}
D_{2n-1}<K^{1/8}\sqrt{D_{n-1}}<K^{5/8}.  \label{dd1}
\end{equation}
In a similar fashion, by (2) in Claim \ref{claim1}, there exists $%
n_{1}>m_{0}+1$ such that $n>n_{1}$ implies
\begin{equation*}
\frac{D_{2n}}{\sqrt{D_{n}}}<(1+\varepsilon)<K^{1/8}
\end{equation*}
and hence
\begin{equation}
D_{2n}<K^{1/8}\sqrt{D_{n}}<K^{5/8}.  \label{dd2}
\end{equation}

Let $n_{2}=\max\{n_{0},n_{1}\}$ and $m_{1}=2n_{2}+1$. It is follows from (%
\ref{dd1}) and (\ref{dd2}) that $n>m_{1}$ implies
\begin{equation}
1\leq D_{n}<K^{5/8}.  \label{dd33}
\end{equation}

To see this let $n>m_{1}$ and let $q$ be a positive integer such that $n=2q$
if $n$ is even and $n=2q-1$ if $n$ is odd. Now if $n$ is odd we have $%
2q-1=n>m_{1}=2n_{2}+1$, and therefore $q>n_{2}>n_{0}$ and it follows from (%
\ref{dd1}) that
\begin{equation*}
D_{n}=D_{2q-1}<K^{5/8}.
\end{equation*}
And if $n$ is even we have $2q=n>m_{1}=2n_{2}+1$, and a it follows from (\ref%
{dd2}) that
\begin{equation*}
D_{n}=D_{2q}<K^{5/8}.
\end{equation*}
\end{proof}

\begin{claim}
\label{claim3} Let $L>1$. Suppose that there exists $m_{0}$ such that $%
n>m_{0}$ implies $1\leq D_{n}<L$. Then there exists $m_{2}\geq m_{0}$ such
that $n>m_{2}$ implies
\begin{equation}
1\leq D_{n}<\sqrt{L}.  \label{ssq}
\end{equation}
\end{claim}

\begin{proof}
From Claim \ref{claim2}, with $K=L$, we know that there exists a positive
integer $m_{1}$ so that $1\leq D_{n}<L^{5/8}$ for all $n>m_{1}.$ Using again
Claim \ref{claim2}, this time with $K=L^{5/8}$, we can assure the existence
of a positive integer $m_{2}$ such that
\begin{equation*}
1\leq D_{n}<\left( L^{5/8}\right) ^{5/8}
\end{equation*}%
for all $n>m_{2}$. Since $\left( \frac{5}{8}\right) ^{2}<\frac{1}{2}$ we
have (\ref{ssq}).
\end{proof}

A recursive application of the above result together with the upper bound in
(\ref{di1}) allows us to obtain a better upper bound (which approaches $1$)
as $n$ goes to infinity:

\begin{claim}
\label{claim4} Let $C$ be an upper bound for $\left(D_{n}\right)_{n=1}^{%
\infty}$ in equation \eqref{di1}. For every non negative integer $s$ there
exists a positive integer $n_{0}$ (depending on $s$) such that $n>n_{0}$
implies
\begin{equation*}
1\leq D_{n}<C^{2^{-s}}.
\end{equation*}
\end{claim}

\begin{proof}
We argue by induction on $s$. The case $s=0$ is equation \eqref{di1}.
Suppose that the result holds for some $s\geq0$. Then there exists $n_{0}$
such that $n>n_{0}$ implies $1\leq D_{n}<C^{2^{-s}}$ and, by Claim \ref%
{claim3}, there exists $m_{1}>n_{0}$ such that
\begin{equation*}
1\leq D_{n}<\left( C^{2^{-s}}\right)^{1/2}=C^{2^{-(s+1)}},
\end{equation*}
whenever $n>m_{1}$.
\end{proof}

After the four previous claims, we can now easily conclude that $\left(
D_{n}\right) _{n=1}^{\infty }$ is convergent and $\displaystyle%
\lim_{n\rightarrow \infty }D_{n}=1$. In fact, let $\varepsilon >0;$ let $%
s_{0}$ be a positive integer such that $C^{\left( 2^{-s_{0}}\right)
}<1+\varepsilon$ and the result follows from Claim \ref{claim4}.
% that there exists $n_{1}
%$ such that $n>n_{1}$ implies $1\leq D_{n}<C^{\left( 2^{-s_{0}}\right) }$. Thus, $1\leq D_{n}<1+\varepsilon ,$whenever $n>n_{1}$, obtaining the %desired
%result.

%%%%%%%%%%%%
%%%%%%%%%%%%
%%% REFERENCES
%%%%%%%%%%%%
%%%%%%%%%%%%

\end{document}